\documentclass[11pt,reqno]{amsproc}

\title[]{Inviscid limit for SQG in bounded domains}

\author{Peter Constantin}
\address{Department of Mathematics, Princeton University, Princeton, NJ 08544}
\email{const@math.princeton.edu}

\author{Mihaela Ignatova}
\address{Department of Mathematics, Princeton University, Princeton, NJ 08544}
\email{ignatova@math.princeton.edu}

\author{Huy Q. Nguyen}
\address{Department of Mathematics, Princeton University, Princeton, NJ 08544}
\email{qn@math.princeton.edu}

\usepackage[margin=1in]{geometry}
\usepackage{amsmath, amsthm, amssymb, mathrsfs}
\usepackage{times}
\usepackage{color}
\usepackage{hyperref}

\newcommand{\bq}{\begin{equation}}
\newcommand{\eq}{\end{equation}}
\newcommand{\bqa}{\begin{eqnarray*}}
\newcommand{\eqa}{\end{eqnarray*}}
\newcommand{\Rr}{\mathbb{R}}

\newcommand{\la}{\label}

\newcommand{\na}{\nabla}
\newcommand{\be}{\begin{equation}}
\newcommand{\ee}{\end{equation}}
\newcommand{\ba}{\begin{array}{l}}
\newcommand{\ea}{\end{array}}

\theoremstyle{plain}
\newtheorem{theo}{Theorem}[section]

\newtheorem{lemm}[theo]{Lemma}
\newtheorem{coro}[theo]{Corollary}

\theoremstyle{definition}
\newtheorem{rema}[theo]{Remark}

\DeclareMathOperator{\cnx}{div}

\DeclareSymbolFont{pletters}{OT1}{cmr}{m}{sl}
\DeclareMathSymbol{s}{\mathalpha}{pletters}{`s}

\def\tt{\theta}

\def\eps{\varepsilon}

\def\na{\nabla}
\def\la{\left\lvert}

\def\le{\leq}

\def\L#1{\langle #1 \rangle}

\def\mez{\frac{1}{2}}

\def\ra{\right\rvert}

\def\P{\mathbb P}

\def\L{\Lambda}

\def\p{\partial}
\def\wc{\rightharpoonup}

\numberwithin{equation}{section}

\date{today}
\begin{document}
\begin{abstract}
We prove that the limit of any weakly convergent sequence of Leray-Hopf solutions of dissipative SQG equations is a weak solution of the inviscid SQG equation in bounded domains.  
\end{abstract}


\noindent

\maketitle
\section{Introduction}
The behavior of high Reynolds number fluids is a broad, important and mostly open problem of nonlinear physics and of PDE. Here we consider a model problem, 
the surface quasi-geostrophic equation, and the limit of its viscous regularizations of certain types. We prove that the inviscid limit is rigid, and no anomalies arise in the limit.  

Let $\Omega\subset \Rr^2$ be a bounded domain with smooth boundary. Denote 
\[
\L=\sqrt{-\Delta}
\]
where $-\Delta$ is the Laplacian operator with Dirichlet boundary conditions. The dissipative  surface quasigeostrophic (SQG) equation  in $\Omega$ is the equation
\bq\label{SQG}
\partial_t\tt^\nu  +u^\nu\cdot \nabla\tt^\nu +\nu \Lambda^s\theta^\nu=0, \quad\nu>0,~s\in (0, 2], 
\eq
where $\tt^\nu = \tt^\nu(x,t)$, $u^\nu = u^\nu(x,t)$ with $(x, t)\in \Omega\times [0, \infty)$ and with the velocity $u^\nu$ given by
\bq\label{u:defi}
u^\nu= R_D^\perp \tt^\nu :=\nabla^\perp \L^{-1}\tt^\nu,\quad\nabla^\perp=(-\p_2, \p_1).
\eq
We refer to the parameter $\nu$ as  ``viscosity''. Fractional powers of the Laplacian $-\Delta$ are based on eigenfunction expansions. The inviscid SQG equation has zero viscosity
\bq\label{iSQG}
\partial_t\tt  +u\cdot \nabla\tt =0,\quad u=R_D^\perp \theta.
\eq
The dissipative  SQG \eqref{SQG} has global weak solutions for any $L^2$ initial data:
\begin{theo}\label{globalweak}
For any initial data $\theta_0\in L^2(\Omega)$ there exists a global weak solution $\theta$
\[
\theta\in C_w(0, \infty; L^2(\Omega))\cap L^2(0, \infty; D(\L^{\frac s2}))
\]
 to the dissipative  SQG equation \eqref{SQG}. More precisely,  $\theta$ satisfies the weak formulation
 \bq\label{weakform:dissipative }
 \int_0^\infty\int_\Omega \theta \varphi(x)dx \p_t \phi(t)dt+\int_0^\infty\int_\Omega u\theta \cdot \nabla\varphi(x)dx \phi(t) dt-\nu\int_0^\infty\int_\Omega \L^{\frac s2}\tt \L^{\frac s2}\varphi(x)dx \phi(t)dt=0
 \eq
 for any $\phi\in C_c^\infty((0, \infty))$ and $\varphi\in D(\L^2)$. Moreover, $\tt$ obeys the energy inequality
 \bq\label{energyineq}
\mez\Vert \tt(\cdot, t)\Vert^2_{L^2(\Omega)}+\nu\int_0^t\int_\Omega |\L^{\frac s2}\theta|^2dxdr\le \mez \Vert \tt_0\Vert^2_{L^2(\Omega)}
 \eq
and the  balance 
 \bq\label{Ham:viscous}
\mez\Vert \tt(\cdot, t)\Vert^2_{D(\L^{-\mez})}+\nu\int_0^t\int_\Omega |\L^{\frac {s-1}2}\theta|^2dxdr=\mez\Vert \tt_0\Vert^2_{D(\L^{-\mez})}
\eq
  for a.e. $t>0$. In addition, $\tt\in C([0, \infty); D(\L^{-\eps}))$ for any $\eps>0$ and the initial data $\tt_0$ is attained in $D(\L^{-\eps})$.
\end{theo}
We refer to any weak solutions of \eqref{SQG} satisfying the properties \eqref{weakform:dissipative }, \eqref{energyineq}, \eqref{Ham:viscous} as a ``Leray-Hopf weak solution''.
\begin{rema}
Theorem \ref{globalweak}  for critical dissipative SQG $s=1$ was obtained in \cite{ConIgn}.
\end{rema}
\begin{rema}
Note that $C_c^\infty(\Omega)$ is not dense in $D(\L^2)$ since the $D(\L^2)$ norm is equivalent to the $H^2(\Omega)$ norm and $C_c^\infty(\Omega)$ is dense in $H^2_0(\Omega)$ which is strictly contained in $D(\L^2)$. 
\end{rema}
The existence of $L^2$ global weak solutions for inviscid SQG \eqref{iSQG}  was proved in \cite{ConNgu}. More precisely, (see Theorem 1.1, \cite{ConNgu}) for any initial data $\tt_0\in L^2(\Omega)$  there exists a global weak solution $\tt\in C_w(0, \infty; L^2(\Omega))$ satisfying
\bq\label{weak:iSQG}
 \int_0^\infty\int_\Omega \theta \p_t\varphi dx dt+\int_0^\infty\int_\Omega u\theta \cdot \nabla\varphi dxdt=0\quad\forall \varphi\in C_c^\infty(\Omega\times (0, \infty)),
 \eq
 and such that the Hamiltonian 
 \bq\label{Ham}
 H(t):=\| \tt(t)\|_{D(\L^{-\mez})}^2
 \eq
  is constant in time. Moreover, the initial data is attained in $D(\L^{-\eps})$ for any $\eps>0$.
 
 Our main result in this note establishes the convergence of weak solutions of the dissipative  SQG to weak solutions of the inviscid SQG in the inviscid limit $\nu\to 0$.
\begin{theo}\label{main}
Let $\{\nu_n\}$ be a sequence of viscosities converging to $0$  and let $\{\theta^{\nu_n}_0\}$ be a bounded sequence in $L^2(\Omega)$. Any  weak limit $\tt$ in $L^2(0, T; L^2(\Omega))$, $T>0$, of any subsequence of $\{\theta^{\nu_n}\}$ of Leray-Hopf weak solutions of the dissipative  SQG equation \eqref{SQG} with viscosity $\nu_n$ and initial data $\theta_0^{\nu_n}$ is a weak solution of the inviscid SQG equation \eqref{iSQG} on $[0, T]$. Moreover, $\tt\in C(0, T; D(\L^{-\eps}))$ for any $\eps>0$, and when $s\in (0, 1]$ the Hamiltonian of $\tt$ is constant on $[0, T]$. 
\end{theo}
\begin{rema}
The same result holds true on the torus $\mathbb{T}^2$. The case of the whole space $\Rr^2$ was treated in \cite{Ber}.
\end{rema}
\begin{rema}
With more singular constitutive laws  $u=\nabla^\perp \L^{-\alpha}\tt$, $\alpha\in [0, 1)$, $L^2$ global weak solutions of the inviscid equations were obtained in \cite{cccgw, HN}. Theorem \ref{main} could be extended to this case. It is also possible to consider $L^p$ initial data in light of the work \cite{Marchand}.
\end{rema}
As a corollary of the proof of Theorem \ref{main} we have the following weak rigidity of inviscid SQG in bounded domains:
\begin{coro}
Any weak limit in $L^2(0, T; L^2(\Omega))$, $T>0$, of any sequence of weak solutions of the inviscid SQG equation \eqref{iSQG} is a weak solution of \eqref{iSQG}. Here, weak solutions of \eqref{iSQG} are interpreted in the sense of \eqref{weak:iSQG}. 
\end{coro}
\begin{rema}
On tori, this result was proved in \cite{IV}. If the weak limit occurs in $L^\infty(0, T; L^2(\Omega))$ and the sequence of weak solutions conserves the Hamiltonian then so is the limiting weak solution.
\end{rema}
The paper is organized as follows. Section \ref{prelim} is devoted to basic facts about the spectral fractional Laplacian and results on commutator estimate. The proofs of Theorems \ref{globalweak} and \ref{main} are given respectively in sections \ref{proof:globalweak} and \ref{proof:main}. Finally, an auxiliary lemma is given in Appendix \ref{app}.
\section{Fractional Laplacian and commutators}\label{prelim}
Let $\Omega\subset \Rr^d$, $d\ge 2$, be a bounded domain  with smooth boundary. The Laplacian $-\Delta$ is defined on $D(-\Delta)=H^2(\Omega)\cap H^1_0(\Omega)$. Let $\{w_j\}_{j=1}^\infty$ be an orthonormal basis of $L^2(\Omega)$ comprised of $L^2-$normalized eigenfunctions $w_j$ of $-\Delta$, i.e.
\[
-\Delta w_j=\lambda_jw_j, \quad \int_{\Omega}w_j^2dx = 1,
\]
with $0<\lambda_1<\lambda_2\le...\le\lambda_j\to \infty$.\\
The fractional Laplacian is defined using eigenfunction expansions,
\[
\Lambda^{s}f\equiv (-\Delta)^{\frac{s}{2}} f:=\sum_{j=1}^\infty\lambda_j^{\frac{s}{2}} f_j w_j\quad\text{with}~f=\sum_{j=1}^\infty f_jw_j,\quad f_j=\int_{\Omega} fw_jdx
\]
for $s\ge 0$ and $f\in D(\Lambda^{s}):=\{f\in L^2(\Omega): \big(\lambda_j^{\frac{s}{2}} f_j\big)\in \ell^2(\mathbb N)\}$. The norm of $f$ in $D(\Lambda^{s })$ is defined by
\[
\Vert f\Vert_{D(\L^s)}:=\|(\lambda_j^{\frac s2}f_j)\|_{\ell^2(\mathbb{N})}.
\]
It is also well-known that $D(\Lambda)$ and $H^1_0(\Omega)$ are isometric. 
In the language of interpolation theory, 
\[
D(\Lambda^\alpha)=[L^2(\Omega), D(-\Delta)]_{\frac \alpha 2}\quad\forall \alpha\in [0, 2].
\]
As mentioned above,
\[
H^1_0(\Omega)= D(\Lambda)=[L^2(\Omega), D(-\Delta)]_{\mez},
\]
hence
\[
D(\Lambda^\alpha)=[L^2(\Omega), H^1_0(\Omega)]_{\alpha}\quad\forall \alpha\in [0, 1].
\]
Consequently, we can identify $D(\Lambda^\alpha)$ with usual Sobolev spaces (see Chapter 1, \cite{LioMag}):
\bq\label{identify}
D(\Lambda^\alpha)=
\begin{cases}
H^\alpha_0(\Omega) &\quad\text{if}~ \alpha\in (\mez, 1],\\
H^\mez_{00}(\Omega):=\{ u\in H^\mez_0(\Omega): u/\sqrt{d(x)}\in L^2(\Omega)\}&\quad\text{if}~ \alpha=\mez,\\
H^\alpha(\Omega) &\quad\text{if}~ \alpha\in [0, \mez).
\end{cases}
\eq
Here and below $d(x)$ denote the distance from $x$ to the boundary $\p\Omega$.

Next, for $s>0$ we define
\[
\L^{-s}f=\sum_{j=1}^\infty \lambda_j^{-\frac{s}{2}}f_jw_j
\]
 if $f=\sum_{j=1}^\infty f_jw_j\in D(\L^{-s})$ where
\[
D(\L^{-s}):=\left\{\sum_{j=1}^\infty f_jw_j\in \mathscr{D}'(\Omega): f_j\in \Rr,~\sum_{j=1}^\infty \lambda_j^{-\frac{s}{2}}f_jw_j\in L^2(\Omega)\right\}.
\]
The norm of $f$ is then defined by
\[
 \Vert f\Vert_{D(\L^{-s})}:=\Vert \L^{-s}f\Vert_{L^2(\Omega)}=\big(\sum_{j=1}^\infty\lambda_j^{-s} f_j^2\big)^\mez.
\]
It is easy to check that $D(\L^{-s})$ is the dual of $D(\L^s)$ with respect to the pivot space $L^2(\Omega)$.
\begin{lemm}[\protect{Lemma 2.1, \cite{HN}}] \label{lemm:inject}
The embedding 
\bq\label{inject}
D(\L^s)\subset H^s(\Omega)
\eq
is continuous for all $s\ge 0$.
\end{lemm}
\begin{lemm}\label{lemm:compact}
For $s, r\in \Rr$ with $s>r$, the embedding $D(\L^s)\subset D(\L^r)$ is compact.
\end{lemm} 
\begin{proof}
Let $\{u_n\}$ be a bounded sequence in $D(\L^s)$. Then $\{\L^ru_n\}$ is bounded in $D(\L^{s-r})$. Choosing $\delta>0$ smaller than $\min(s-r, \mez)$ we have $D(\L^{s-r})\subset D(\L^\delta)=H^\delta(\Omega)\subset L^2(\Omega)$ where the first embedding is continuous and the second is compact. Consequently the embedding $D(\L^{s-r})\subset L^2(\Omega)$ is compact and thus there exist a subsequence $n_j$ and a function $f\in L^2(\Omega)$ such that $\L^ru_{n_j}$ converge to $f$ strongly in $L^2(\Omega)$. Then $u_{n_j}$ converge to $u:=\L^{-r}f$ strongly in $D(\L^r)$ and the proof is complete.
\end{proof}
A bound for the commutator between $\Lambda$ and multiplication by a smooth function was proved in \cite{ConIgn} using the method of harmonic extension:
\begin{theo}[\protect{Theorem 2, \cite{ConIgn}}]\label{Commutator:CI}
Let $\chi\in B(\Omega)$ with $B(\Omega)=W^{2, d}(\Omega)\cap W^{1, \infty}(\Omega)$ if $d\ge 3$, and $B(\Omega)=W^{2, p}(\Omega)$ with $p>2$ if $d=2$. There exists a constant $C(d, p, \Omega)$ such that
\[
\Vert [\Lambda, \chi]\psi\Vert_{D(\L^\mez)}\le C(d, p, \Omega)\Vert \chi\Vert_{B(\Omega)}\Vert \psi\Vert_{D(\L^\mez)}.
\] 
\end{theo}
Pointwise estimates for the commutator between fractional Laplacian and differentiation were established in \cite{ConNgu}:
\begin{theo}[\protect{Theorem 2.2, \cite{ConNgu}}] For any $p\in [1, \infty]$ and $s\in (0, 2)$ there exists a positive constant $C(d, s, p, \Omega)$ such that for all $\psi\in C_c^\infty(\Omega)$ we have
\[
\la [\Lambda^s, \nabla]\psi(x)\ra\le C(d, s, p, \Omega)d(x)^{-s-1-\frac dp}\Vert \psi\Vert_{L^p(\Omega)}
\]
holds for all $x\in \Omega$.
\end{theo}
This pointwise bound implies the following commutator estimate in Lebesgue spaces.
\begin{theo}\label{Commutator:CN}
Let $p,~q\in [1, \infty]$, $s\in (0, 2)$ and $\varphi$ satisfy 
\[
\varphi(\cdot)d(\cdot)^{-s-1-\frac dp}\in L^q(\Omega).
\]
Then the operator $\varphi[\Lambda^s, \nabla]$ can be uniquely extended from $C_c^\infty(\Omega)$ to $L^p(\Omega)$ such that there exists a positive constant $C=C(d, s, p, \Omega)$ such that
\bq\label{commu:CN:e}
\Vert \varphi[\Lambda^s, \nabla] \psi\Vert_{L^q(\Omega)}\le C\Vert \varphi(\cdot)d(\cdot)^{-s-1-\frac dp}\Vert_{L^q(\Omega)}\Vert \psi\Vert_{L^p(\Omega)}
\eq
holds for all $\psi\in L^p(\Omega)$.
\end{theo}
The inequality \eqref{commu:CN:e} is remarkable because the commutator between an operator of order $s\in (0, 2)$ and an operator of order $1$ is an operator of order $0$.

\section{Proof of Theorem \ref{globalweak}}\label{proof:globalweak}
We use Galarkin approximations. Denote by $\P_m$ the projection in $L^2(\Omega)$ onto the linear span $L^2_m$ of eigenfunctions $\{w_1,...,w_m\}$, i.e.
\bq\label{def:Pm}
\P_m f=\sum_{j=1}^mf_jw_j\quad\text{for}~f=\sum_{j=1}^\infty f_jw_j.
 \eq
The $m$th Galerkin approximation of \eqref{SQG} is the following ODE system in the finite dimensional space $ L^2_m$:
\bq\label{Galerkin}
\begin{cases}
\dot \tt_m+\P_m(u_m\cdot\nabla\tt_m)+\nu\L^s\theta_m=0&\quad t>0,\\
\tt_m=P_m\tt_0&\quad t=0,
\end{cases}
\eq
with $\tt_m(x, t)=\sum_{j=1}^m\tt_ j^{(m)} (t)w_j(x)$ and $u_m={R_D}^\perp\tt_m$  satisfying $\cnx u_m=0$. 
Note that \eqref{Galerkin} is equivalent to
\bq
\frac{d\theta^{(m)}_l}{dt} + \sum_{j,k=1}^m\gamma^{(m)}_{jkl}\theta^{(m)}_j\theta^{(m)}_{k} +\nu \lambda_l^{\frac s2}\theta^{(m)}_l=0,\quad l=1,2,...,m,
\label{galmode}
\eq
with
\[
\gamma^{(m)}_{jkl} = \lambda_j^{-\frac{1}{2}}\int_{\Omega}\left(\na^{\perp}w_j\cdot\na w_k\right)w_ldx.
\]
The local existence of $\theta_m$ on some time interval $[0, T_m]$  follows from the Cauchy-Lipschitz theorem. On the other hand, the antisymmetry property $\gamma^{(m)}_{jkl}=-\gamma^{(m)}_{jlk}$ yields
\bq\label{L^2bound}
\mez\Vert \tt_m(\cdot, t)\Vert^2_{L^2(\Omega)}+\nu\int_0^t\int_\Omega |\L^{\frac s2}\theta_m|^2dxdr=\mez\Vert \P_m\tt_0\Vert^2_{L^2(\Omega)}\le \mez \Vert \tt_0\Vert^2_{L^2(\Omega)}
\eq
for all $t\in [0, T_m]$. This implies that $\theta_m$ is global and \eqref{L^2bound} holds for all positive times. The sequence $\theta_m$ is thus uniformly bounded in $L^\infty(0, \infty; L^2(\Omega))\cap L^2(0, \infty; D(\L^{\frac s2}))$.  Upon extracting a subsequence, we have $\theta_m$ converge to some $\theta$ weakly-* in $L^\infty(0, \infty; L^2(\Omega))$ and weakly in $L^2(0, \infty; D(\L^{\frac s2}))$. In particular, $\tt$ obeys the same energy inequality as in \eqref{L^2bound}. On the other hand, if one multiplies \eqref{galmode} by $\lambda_l^{-1/2}\tt_l^{(m)}$ and uses the fact that $\gamma^{(m)}_{jkl}\lambda_l^{-1/2}=-\gamma^{(m)}_{lkj}\lambda_j^{-1/2}$, one obtains 
\bq\label{Ham:m}
\mez\Vert \tt_m(\cdot, t)\Vert^2_{D(\L^{-\mez})}+\nu\int_0^t\int_\Omega |\L^{\frac {s-1}2}\theta_m|^2dxdr=\mez\Vert \P_m\tt_0\Vert^2_{D(\L^{-\mez})}.
\eq
We derive next a uniform bound for $\p_t\tt_m$.  Let $N>0$ be an integer to be determined. For any $\varphi \in D(\L^{2N})$ we integrate by parts to get
\begin{align*}
 \int_\Omega \p_t\tt_m\varphi dx=-&\int_\Omega \P_m \cnx(u_m\tt_m) \varphi dx-\int_\Omega\nu\L^s\tt_m\varphi dx\\
 &=\int_\Omega  (u_m\tt_m) \cdot \nabla(\P_m\varphi)dx-\int_\Omega\nu\tt_m\L^s\phi dx.
 \end{align*}
The first term is controlled by
\[
\left|\int_\Omega  (u_m\tt_m) \cdot \nabla(\P_m\varphi)dx\right|\le \| u_m\tt_m\|_{L^1(\Omega)}\|\nabla\P_m\varphi\|_{L^\infty(\Omega)}\le C\| \P_m\varphi\|_{H^3(\Omega)}.
\]
According to  Lemma \ref{lemm:Pm}, for $N$ and $k$ satisfying $N>\frac{k}{2}+1$ there exists a positive constant $C_{N, k}$ such that
\bq
\| \P_m\varphi\|_{H^k(\Omega)}\le C_{N, k} \| \varphi\|_{D(\L^{2N})}\quad\forall m\ge 1,~\forall\varphi\in D(\L^{2N}).
\eq
 With $k=3$ and $N=3$ we have
\[
\left|\int_\Omega  (u_m\tt_m) \cdot \nabla(\P_m\varphi)dx\right|\le C\|\varphi\|_{D(\L^{6})}.
\]
On the other hand,
\[
\left| \int_\Omega\nu \tt_m\L^{s}\varphi dx\right|\le C\| \tt_m\|_{L^2(\Omega)}\| \varphi\|_{D(\L^2)}.
\]
We have proved that 
\[
\left| \int_\Omega \p_t\tt_m\varphi dx\right|\le C\| \varphi\|_{D(\L^6)}\quad\forall \varphi \in D(\L^6).
\]
Because $L^2(\Omega)\times D(\L^6)\ni (f, g)\mapsto \int_\Omega fgdx$ extends uniquely to a bilinear from on $D(\L^{-6})\times D(\L^6)$, we deduce that $\p_t\tt_m$ are uniformly bounded in $L^\infty(0, \infty; D(\L^{-6}))$. Note that we have used only the uniform regularity $L^\infty(0; \infty; L^2(\Omega))$ of $\tt_m$. We have the embeddings $D(\Lambda^{\frac s2})\subset D(\L^{(s-1)/2})\subset D(\L^{-6})$ where the first one is compact by virtue of Lemma \ref{lemm:compact}, and the second is continuous. Fix $T>0$. Aubin-Lions' lemma (see \cite{Lions}) ensures that for some function $f$ and along some subsequence $\tt_m$ converge to $f$ weakly in $L^2(0, T; D(\L^{\frac s2}))$ and strongly in $L^2(0, T; D(\L^{(s-1)/2}))$. In principle, both $f$ and the subsequence might depend on $T$, however, we already know that $\tt_m\to\tt$ weakly in $L^2(0, \infty; D(\L^{\frac s2}))$. Therefore, $f=\tt$ and the convergences to $\tt$ hold for the whole sequence. Similarly, applying Aubin-Lions' lemma with the embeddings $L^2(\Omega)\subset D(\L^{-\eps})\subset D(\L^{-6})$ for sufficiently small $\eps>0$ we obtain that $\tt_m\to \tt$ strongly in $C([0, T];  D(\L^{-\eps}))$. Integrating \eqref{Galerkin} against an arbitrary test function of the form $\phi(t)\varphi(x)$ with $\phi\in C_c^\infty((0, T))$, $\varphi\in D(\L^6)$ yields
\[
 \int_0^T\int_\Omega \theta_m  \varphi(x) dx\p_t\phi(t) dt+\int_0^T\int_\Omega u_m\theta_m \cdot \nabla\P_m\varphi(x)dx\phi(t) dt-\nu\int_0^T\int_\Omega \L^{\frac s2}\tt_m \L^{\frac s2}\varphi(x)dx\phi(t)dt=0.
\]
By Lemma \ref{lemm:Pm},
\[
\|(\mathbb{I}-\P_m)\varphi\|_{L^\infty(\Omega)}\le C\|(\mathbb{I}-\P_m)\varphi\|_{H^3(\Omega)}\to 0\quad\text{as}~m\to \infty.
\]
 The weak convergence of $\tt_m$ in $L^2(0, T; D(\L^{\frac s2}))$ allows one to pass to the limit in the two linear terms. The strong convergence of $\tt_m$ in $L^2(0, T; L^2(\Omega))$ together with the weak convergence of $u_m$ in the same space allows one to pass to the limit in the nonlinear term  and conclude that $\tt$ satisfies the weak formulation \eqref{weakform:dissipative } with $\varphi \in D(\L^6)$.  In fact, $\tt\in L^2(0, \infty; D(\L^{\frac s2}))\subset L^2(0, \infty; L^p(\Omega))$ for some $p>2$, hence $u\tt\in L^2(0, \infty; L^q(\Omega))$ for some $q>1$.  In addition, if $\varphi\in D(\L^2)$ then $\nabla \varphi\in L^r$ for all $r<\infty$, and thus the nonlinearity $\int_\Omega u\theta\cdot \nabla \varphi dx$ makes sense. Then because $D(\L^2)$ is dense in $D(\L^6)$, \eqref{weakform:dissipative }  holds for $\varphi \in D(\L^2)$. 
 
 We now pass to the limit in \eqref{Ham:m}. The strong convergence $\tt_m\to \tt$ in $C(0, T;  D(\L^{-\eps}))$ gives the convergence of the first term. On the other hand, the strong convergence $\tt_m\to \tt$ in $L^2(0, T; D(\L^{(s-1)/2}))$ yields the convergence of the second term. The right hand side converges to $\mez\|\tt_0\|_{D(\L^{-\mez})}^2$ since $\P_m\tt_0$ converge to $\tt_0$ in $L^2(\Omega)$. We thus obtain \eqref{Ham:viscous}.
 
 Since $\tt_m\to \tt$ in $C([0, T];  D(\L^{-\eps}))$ we deduce that
\[
\tt_0=\lim_{m\to \infty}\P_m\tt_0=\lim_{m\to \infty}\tt_m\vert_{t=0}=\tt\vert_{t=0}\quad\text{in}~  D(\L^{-\eps}).
\]
For a.e. $t\in [0, T]$, $\tt_m(t)$ are uniformly bounded in $L^2(\Omega)$, and thus along some subsequence $m_j$, a priori depending on $t$, we have $\tt_{m_j}(t)$ converge weakly to some $f(t)$ in $L^2(\Omega)$. But we know $\tt_m(t)\to \tt(t)$ in $D(\L^{-\eps})$. Thus, $f(t)=\tt(t)$ and $\tt_m(t)\wc \tt(t)$ in $L^2(\Omega)$ as a whole sequence for a.e. $t\in [0, T]$. 
Recall that $\frac{d}{dt}\tt_m$ are uniformly bounded in $L^\infty(0, T; D(\L^{-6}))$. For all $\varphi\in D(\L^6)$ and $t\in [0, T]$ we write 
\[
\langle \tt_m(t), \varphi\rangle_{L^2(\Omega), L^2(\Omega)}=\langle \tt_m(0), \varphi\rangle_{L^2(\Omega), L^2(\Omega)}+\int_0^t\langle \frac{d}{dt}\tt_m(r), \varphi\rangle_{D(\L^{-6}), D(\L^6)} dr.
\]
Because $\frac{d}{dt}\tt_m$ converge to $\frac{d}{dt}\tt$ weakly-* in $L^\infty(0, T; D(\L^{-6}))$, letting $m\to \infty$ yields
\[
\langle \tt(t), \varphi\rangle_{L^2(\Omega), L^2(\Omega)}=\langle \tt_0, \varphi\rangle_{L^2(\Omega), L^2(\Omega)}+\int_0^t\langle \frac{d}{dt}\tt(r), \varphi\rangle_{D(\L^{-6}), D(\L^6)} dr
\]
for a.e. $t\in [0, T]$. Taking the limit $t\to 0$ gives
\[
\lim_{t\to 0}\langle \tt(t), \varphi\rangle_{L^2(\Omega), L^2(\Omega)}=\langle \tt_0, \varphi\rangle_{L^2(\Omega), L^2(\Omega)}
\]
for all $\varphi\in D(\L^6)$. Finally, since $ D(\L^6)$ is dense in $L^2(\Omega)$ and $\tt\in L^\infty(0, T; L^2(\Omega))$ we conclude that $\tt\in C_w(0, T; L^2(\Omega))$ for all $T>0$. 
 \section{Proof of Theorem \ref{main}}\label{proof:main}
First, using  approximations and commutator estimates we justify the commutator structure of the SQG nonlinearity derived in \cite{ConNgu}.
\begin{lemm} \label{commu:key}
For all $\psi \in H^1_0(\Omega)$ and $\varphi\in C_c^{\infty}(\Omega)$ we have
\bq\label{key}
\int_{\Omega}\Lambda \psi\nabla^\perp\psi\cdot \nabla \varphi dx=\mez \int_\Omega  [\Lambda, \nabla^\perp]\psi\cdot \nabla\varphi\psi dx-\mez\int_\Omega  \nabla^\perp\psi\cdot [\Lambda, \nabla\varphi] \psi dx.
\eq
Here, the commutator $[\Lambda, \nabla^\perp]\psi\cdot \nabla\varphi$ is understood in the sense of the extended operator defined in Theorem \ref{Commutator:CN}.
\end{lemm}
\begin{proof}
Let $\psi_n\in C_c^\infty(\Omega)$ converging to $\psi $ in $H^1_0(\Omega)$. Integrating by parts and using the fact that $\na^{\perp}\cdot\na\varphi = 0$ gives
\[
\int_{\Omega}\Lambda \psi_n\nabla^\perp\psi_n\cdot \nabla \varphi dx  = -\int_{\Omega}\psi_n\nabla^\perp\L\psi_n\cdot \nabla \varphi dx,
\]
 Because $\psi_n$ is smooth and has compact support, $\nabla^\perp \psi_n\in D(\L)$, and thus we can commute $\nabla^\perp$ with $\L$ to obtain
\[
\begin{aligned}&
\int_{\Omega}\Lambda \psi_n\nabla^\perp\psi_n\cdot \nabla \varphi dx\\
&=-\int_\Omega   \psi_n [\nabla^\perp,\Lambda]\psi_n \cdot \nabla\varphi  dx-\int_\Omega   \psi_n\Lambda \nabla^\perp\psi_n\cdot \nabla\varphi dx\\
&=-\int_\Omega  \psi_n [\nabla^\perp,\Lambda]\psi_n\cdot \nabla\varphi dx-\int_\Omega  \nabla^\perp\psi_n\cdot \Lambda(\psi_n \nabla\varphi) dx\\
&=-\int_\Omega  [\nabla^\perp,\Lambda]\psi_n\cdot \nabla\varphi\psi_n dx-\int_\Omega  \nabla^\perp\psi_n\cdot [\Lambda, \nabla\varphi] \psi_n dx-\int_\Omega  \nabla^\perp\psi_n\cdot \nabla \varphi \Lambda\psi_n dx.
\end{aligned}
\]
Noticing that the last term on the right-hand side is exactly the negative of the left-hand side, we deduce that
\[
\int_{\Omega}\Lambda \psi_n\nabla^\perp\psi_n\cdot \nabla \varphi dx=\mez\int_\Omega  [\Lambda, \nabla^\perp]\psi_n\cdot \nabla\varphi\psi_n dx-\mez\int_\Omega  \nabla^\perp\psi_n\cdot [\Lambda, \nabla\varphi] \psi_n dx.
\]
The commutator estimates in Theorems \ref{Commutator:CI} and \ref{Commutator:CN} then allow us to pass to the limit in the preceding representation and conclude that \eqref{key} holds.
\end{proof}
Now let $\nu_n\to  0^+$ and let $\theta^{\nu_n}_0$ be a bounded sequence in $L^2(\Omega)$. For each $n$ let $\theta_n\equiv \theta^{\nu_n}$ be a  Leray-Hopf weak solution of \eqref{SQG} with viscosity $\nu_n$ and initial data $\theta^{\nu_n}_0$. In view of the energy inequality \eqref{energyineq}, $\tt_n$ are uniformly bounded in $L^\infty(0, \infty; L^2(\Omega))$
 and satisfies
\bq\label{weak:proof}
 \int_0^\infty\int_\Omega \theta_n\varphi(x)dx\p_t\phi(t)dt+\int_0^\infty\int_\Omega u_n\theta_n \cdot \nabla\varphi(x)dx \phi(t)dt-\nu_n\int_0^\infty\int_\Omega \L^{\frac s2}\tt_n \L^{\frac s2}\varphi(x)dx\phi(t) dt=0
 \eq
for all $\phi\in C_c^\infty((0, \infty))$ and $\varphi\in D(\L^2)$. Fix $T>0$. Assume that along a subsequence, still labeled by $n$, $\tt_n$ converge to $\tt$ weakly in $L^2(0, T; L^2(\Omega))$.  We prove that $\tt$ is a weak solution of the inviscid SQG equation. We first prove a uniform bound for $\p_t\tt_n$ provided only the uniform regularity $L^\infty(0, T; L^2(\Omega))$ of $\tt_n$. To this end, let us define for a.e. $t\in [0, T]$ the function $f_n(\cdot ,t)\in H^{-3}(\Omega)$ by 
\[
\langle f_n(t), \varphi\rangle_{H^{-3}(\Omega), H^3_0(\Omega)}:=\int_\Omega (u_n(x, t)\theta_n(x, t)\cdot \nabla\varphi(x)-\nu_n\tt_n(x, t)\L^s\varphi(x))dx
\]
for all $\varphi\in H^3_0(\Omega)\subset D(\L^2)$. Indeed,  we have
\begin{align*}
\left| \int_\Omega (u_n(x, t)\theta_n(x, t)\cdot \nabla\varphi(x)-\nu_n\tt_n(x, t)\L^s\varphi(x))dx\right|&\le C\big(\|\tt_n(t)\|_{L^2(\Omega)}^2+1\big)\| \varphi\|_{H^3(\Omega)}.\end{align*}
 This shows that $f_n$ are uniformly bounded in $ L^\infty(0, T; H^{-3}(\Omega))$. Then for any $\phi\in C_c^\infty((0, T))$, it follows from \eqref{weak:proof} that
\[
\int_0^T\tt_n \p_t\phi dt=-\int_0^T f_n\phi dt.
\]
In other words, $\p_t \tt_n=f_n$ and the desired uniform bound for $\p_t\tt_n$ follows. Fix $\eps\in (0, \mez)$. Aubin-Lions' lemma applied with the embeddings $L^2(\Omega)\subset D(\L^{-\eps})\subset H^{-3}(\Omega)$ then ensures that $\tt_n$ converge to $\tt$ strongly in $C(0, T; D(\L^{-\eps}))\subset C(0, T; H^{-1}(\Omega))$. Consequently $\psi_n$ converge to $\psi:=\L^{-1}\tt$ strongly in $C(0, T; L^2(\Omega))$. 

Now we take $\phi\in C_c^\infty((0, \infty))$ and $\varphi\in C^\infty_c(\Omega)$.  By virtue of Lemma \ref{commu:key}, the weak formulation \eqref{weakform:dissipative } gives
\begin{align*}
 &\int_0^T\int_\Omega \theta_n\varphi(x) dx \p_t\phi(t)dt+\mez\int_0^T\int_\Omega  [\Lambda, \nabla^\perp]\psi_n\cdot \nabla\varphi(x)\psi_n dx\phi(t) dt\\
 &\qquad -\mez\int_0^T\int_\Omega  \nabla^\perp\psi_n\cdot [\Lambda, \nabla\varphi(x)] \psi_ndx \phi(t) dt -\nu_n\int_0^T\int_\Omega \tt_n \L^s\varphi(x)dx \phi(t) dt=0,
 \end{align*}
where $\psi_n:=\L^{-1}\tt_n$ are uniformly bounded in $L^\infty(0, T; H^1_0(\Omega))$.  The weak convergence $\tt_n\wc \tt$ in $L^2(0, T; L^2(\Omega))$ readily yields
\[
\lim_{n\to \infty}\int_0^T\int_\Omega \theta_n \varphi(x)dx \p_t\phi(t) dt=\int_0^T\int_\Omega \theta \varphi(x)dx \p_t\phi(t) dt
\]
and 
\[
\lim_{n\to \infty}\nu_n\int_0^T\int_\Omega \tt_n \L^s\varphi(x) dx \phi(t) dt=0.
\]
Next we pass to the limit in the two nonlinear terms. Applying the commutator estimate in Theorem \ref{Commutator:CI}  we have
\begin{align*}
&\left|\int_0^T\int_\Omega  \nabla^\perp\psi_n\cdot [\Lambda, \nabla\varphi] \psi_n dx \phi dt-\int_0^T\int_\Omega  \nabla^\perp\psi\cdot [\Lambda, \nabla\varphi] \psi dx\phi dt \right|\\
&\le \left|\int_0^T\int_\Omega  \nabla^\perp(\psi_n-\psi)\cdot [\Lambda, \nabla\varphi] \psi dx \phi dt\right|+\|  \phi\nabla^\perp\psi_n\|_{L^2(0, T; L^2(\Omega))}\| [\Lambda, \nabla\varphi] (\psi_n-\psi)\|_{L^2(0, T; L^2(\Omega))}\\
&\le \left|\int_0^T\int_\Omega  \nabla^\perp(\psi_n-\psi)\cdot [\Lambda, \nabla\varphi] \psi dx \phi dt\right|+C\|\psi_n-\psi\|_{L^2(0, T; D(\L^{\mez}))}.
\end{align*}
The first term converges to $0$ due to the weak convergence of $\psi_n$ to $\psi$ in $L^2(0, T; H^1_0(\Omega))$ and the fact that $[\Lambda, \nabla\varphi] \psi\in D(\L^\mez)\subset L^2(\Omega)$ in view of Theorem \ref{Commutator:CI}. By interpolation, the second term is bounded by 
\[
\|\psi_n-\psi\|_{L^2(0, T; D(\L^{\mez}))}\le \|\psi_n-\psi\|_{L^2(0, T; L^2(\Omega))}^\mez\|\psi_n-\psi\|_{L^2(0, T; D(\L))}^\mez\le C\|\psi_n-\psi\|_{L^2(0, T; L^2(\Omega))}^\mez
\]
which also converge to $0$. Finally, we apply the commutator estimate in Theorem \ref{Commutator:CN} to obtain
\begin{align*}
&\left| \int_0^T\int_\Omega  [\Lambda, \nabla^\perp]\psi_n\cdot \nabla\varphi\psi_n dx \phi dt-\int_0^T\int_\Omega  [\Lambda, \nabla^\perp]\psi\cdot \nabla\varphi\psi dx \phi dt\right|\\
&\le \| \nabla\varphi[\Lambda, \nabla^\perp](\psi_n-\psi)\|_{L^2(0, T; L^2(\Omega))}\|\phi\psi_n\|_{L^2(0, T; L^2(\Omega))}\\
&\qquad+\|[\Lambda, \nabla^\perp]\psi\cdot \nabla \varphi\|_{L^2(0, T; L^2(\Omega))}\|\phi(\psi_n-\psi)\|_{L^2(0, T; L^2(\Omega))}\\
&\le C\| \psi_n-\psi\|_{L^2(0, T; L^2(\Omega))}
\end{align*}
which converge to $0$.  Putting together the above considerations leads to
\[
 \int_0^T\int_\Omega \theta \varphi(x)dx \p_t\phi (t) dt+\int_0^T\int_\Omega u\theta \cdot \nabla\varphi(x)dx \phi(t)dt=0,\quad\forall \phi\in C^\infty_c((0, T)),~\varphi\in C_c^\infty(\Omega).
 \]
 Therefore, $\tt$ is a weak solution of the inviscid SQG equation on $[0, T]$. Finally, consider $s\in (0, 1]$. We have the the balance \eqref{Ham:viscous} for each $\tt_n$. Since $s\le 1$ the uniform boundedness of  $\tt_n$ in $L^\infty(0, T; L^2(\Omega))$ implies 
 \[
\lim_{n\to \infty}\nu_n\int_0^t\int_\Omega |\L^{\frac {s-1}2}\theta_n|^2dxdr=0,\quad t\in [0, T].
 \]
 In addition, $\tt_n\to \tt$ strongly in  $C(0, T; D(\L^{-\eps}))\subset C(0, T; D(\L^{-\mez}))$. Letting $\nu=\nu_n\to 0$ in the balance \eqref{Ham:viscous} we conclude that the Hamiltonian of $\tt$ is constant on $[0, T]$. 
\appendix
\section{A bound on $\P_m$}\label{app}
Recall the definition \eqref{def:Pm} of $\P_m$. The following lemma is essentially taken from \cite{ConNgu}. We include the proof for the sake of completeness.
\begin{lemm}\label{lemm:Pm}
Let $\Omega\subset \Rr^d$, $d\ge 2$, be a bounded domain  with smooth boundary. For every $N$ and $k\in {\mathbb{N}}$ satisfying $N>\frac{k}{2}+\frac{d}{2}$ there exists a positive constant $C_{N, k}$ such that
\bq\label{boundPm}
\| \P_m\varphi\|_{H^k(\Omega)}\le C_{N, k}\| \varphi\|_{D(\L^{2N})}
\eq
for all $m\ge 1$ and $\phi\in D(\L^{2N})$; moreover, we have
\bq\label{conv:Pm}
\lim_{m\to \infty}\| (\mathbb{I}-\P_m)\varphi\|_{H^k(\Omega)}=0.
\eq
\end{lemm}
\begin{proof}
As $\varphi\in D(\L^{2N})$, we have $\Delta^\ell\varphi\in H^1_0(\Omega)$ for all $\ell=0, 1,\dots, N-1$. This allows repeated integration by parts with $w_j$ using the relation $-\Delta w_j=\lambda_j w_j$. Using H\"older's inequality and the fact that $w_j$ is normalized in $L^2$, we obtain
\[
|\varphi_j|\le \lambda_j^{-N}\|\Delta^N\varphi\|_{L^2},\quad \varphi_j=\int_\Omega \varphi w_j dx.
\]
By elliptic regularity estimates and induction, we have for all $k\in \mathbb N$ that
\[
\Vert w_j\Vert_{H^k(\Omega)}\le C_k\lambda_j^{\frac k2}.
\]
We know from the easy part of Weyl's asymptotic law that $\lambda_j \ge Cj^{\frac {2}{d}}$. Consequently, with  $N>\frac{k}{2}+\frac{d}{2}$ we deduce that
\begin{align*}
 \sum_{j=1}^\infty |\varphi_j|\| w_j\|_{H^k(\Omega)}&\le C_k\|\Delta^N\varphi\|_{L^2}\sum_{j=1}^\infty\lambda_j^{-N+\frac{k}{2}}\\
 &\le C_k\|\varphi\|_{D(\L^{2N})}\sum_{j=1}^\infty j^{(-N+\frac{k}{2})\frac{2}{d}}\\
 &=C_{N, k}\| \varphi\|_{D(\L^{2N})}
\end{align*}
where $C_{N, k}<\infty$ depends only on $N$ and $k$. Because 
\[
(\mathbb{I}-\P_m)\varphi=\sum_{j=m+1}^\infty \varphi_jw_j,
\]
 this proves both \eqref{boundPm} and \eqref{conv:Pm}. The proof is complete.
\end{proof}
{\bf{Acknowledgment.}} The research of PC was partially supported by NSF grant
DMS-1713985.

\end{document}